\documentclass[11pt]{article}
\usepackage{amsmath}
\usepackage{amsthm, amssymb}
\usepackage{amsfonts}
\usepackage{tcolorbox}
\usepackage{graphicx}
\usepackage{physics}
\usepackage{dsfont}
\usepackage{xcolor}
\usepackage{appendix}

\usepackage{authblk}

\usepackage[a4paper]{geometry}
\title{A note on Puder's generalised co-growth formula for trees}
\author[]{Wenbo Li and Joe Thomas}

\date{}
\begin{document}
\maketitle

\newtheorem{theorem}{Theorem}[section]
\newtheorem{proposition}{Proposition}[section]
\newtheorem{corollary}{Corollary}[section]
\newtheorem{lemma}{Lemma}[section]
\newtheorem{definition}{Definition}[section]
\newtheorem{remark}{Remark}[section]
\newtheorem{eg}{Example}[section]

\begin{abstract}
    In this note, we prove a conjecture of Puder on an extension of the co-growth formula to any non-negative function defined on a bi-regular tree. A key component of our proof is the establishment of a resolvent identity, which serves as an operator version of the co-growth formula. We also provide a simpler proof of Puder's generalised co-growth formula for the regular tree.
\end{abstract}

\section{Introduction}

Let $G=(V,E)$ be a graph and a fix vertex $v\in V$. For any non-zero $f:V\to\mathbb{R}_{\geq 0}$ and $r\in\mathbb{N}$, set
    \begin{align*}
        a_r(f) &= \sum_{\substack{p : \text{ a non-backtracking walk} \\ \text{from $v$ of length $r$}}} f(\mathrm{end}(p)),\\
        b_r(f)&= \sum_{\substack{p : \text{ a walk from $v$} \\ \text{of length $r$}}} f(\mathrm{end}(p)),
    \end{align*}    
where $\mathrm{end}(p)$ is the final vertex of the walk $p$. The exponential growth rates of $a_r$ and $b_r$ are denoted by

$$\alpha(f)=\limsup_{r\to\infty}a_r(f)^{\frac{1}{r}}\ \ \   \text{and} \ \ \ \beta(f)=\limsup_{r\to\infty}b_r(f)^{\frac{1}{r}}.$$

Note that although $a_r(f)$ and $b_r(f)$ depend on the particular choice of vertex $v \in V$, $a(f)$ and $b(f)$ do not, see Appendix \ref{Appendix}.

The classical co-growth formula of Grigorchuk \cite{Grigor1977} and Cohen \cite{Cohen1982}, with an extension by Northshield \cite{Northshield1992} (see \cite[Section 1]{Puder2024} for more extensive references), establishes a fundamental relationship between $\alpha(f)$ and $\beta(f)$ when $G$ is a $d$-regular graph and $f$ is the indicator function on a vertex of $G$ (in which case $\alpha(f)$ is called the \textit{co-growth} of $G$). Equivalently, by lifting to the infinite $d$-regular tree, the co-growth formula relates $\alpha(f)$ and $\beta(f)$ where $f$ is the indicator function on the fibre over the vertex (this is a special case of Theorem \ref{thm:co-growth}). In this case, $\beta(f)$ ($\alpha(f)$) is the exponential growth rate of the number of (non-backtracking) walks starting and ending at $v$ in $G$.

Recently, Puder \cite{Puder2024} has generalised the co-growth formula to indicator functions of arbitrary subsets of (bi)-regular trees which has subsequently been applied in the breakthrough work of Chen, Garza-Vargas, Tropp and van Handel \cite{CVTvH2024} to understand the `staircase behaviour' of outliers in the adjacency matrix spectrum of random regular graphs. Moreover, for the regular tree, Puder proves a stronger version of the co-growth formula valid for all non-negative functions.

\begin{theorem}[{\cite[Theorems 2.1 and 1.7]{Puder2024}}]
\label{thm:co-growth}
    Let $f\neq 0$ be a non-negative function on a $d$-regular tree with $d\geq 3$. Then,
    \begin{equation*}
        \beta(f)=
\begin{cases}
  2\sqrt{d-1}& \text{if }\alpha(f)\leq \sqrt{d-1},\\
  \alpha(f)+\frac{d-1}{\alpha(f)}& \text{if } \alpha(f)>\sqrt{d-1}.
\end{cases}
    \end{equation*}
Moreover, if $f=\mathds{1}_S$ for $\emptyset\neq S\subseteq V(G)$ and $G$ is a $(k,l)$-bi-regular tree, then 
    $$\begin{aligned}
        \beta(f) = \begin{cases}
            \sqrt{k-1}+\sqrt{l-1} & \text{if } \alpha(f)\leq (k-1)^\frac{1}{4}(l-1)^\frac{1}{4},\\
        \left(\alpha(f)+\frac{k-1}{\alpha(f)}\right)^{1/2}\left(\alpha(f)+\frac{l-1}{\alpha(f)}\right)^{1/2} & \text{if } \alpha(f)\geq (k-1)^\frac{1}{4}(l-1)^\frac{1}{4}.
        \end{cases}
    \end{aligned}$$
\end{theorem}

Puder further conjectured \cite[Conjecture 3.1]{Puder2024} that the generalized co-growth formula, that is, replacing $\mathds{1}_S$ with any non-negative function, still persists for bi-regular trees. In this article, we confirm that this is indeed the case, proving the following result.

\begin{theorem}\label{thm:co-growth2}
Let $f\neq 0$ be a non-negative function on a $(k,l)$-bi-regular tree. Then,
    $$\begin{aligned}
        \beta(f) = \begin{cases}
            \sqrt{k-1}+\sqrt{l-1} & \text{if } \alpha(f)\leq (k-1)^\frac{1}{4}(l-1)^\frac{1}{4},\\
        \left(\alpha(f)+\frac{k-1}{\alpha(f)}\right)^{1/2}\left(\alpha(f)+\frac{l-1}{\alpha(f)}\right)^{1/2} & \text{if } \alpha(f)\geq (k-1)^\frac{1}{4}(l-1)^\frac{1}{4}.
        \end{cases}
    \end{aligned}$$
\end{theorem} 

Much effort has been made to generalise results from regular graphs to bi-regular graphs due to the similarity of their graph structure, for example, on the study of mixing rate of random walks \cite{Kempton2016} and optimal spectral gap \cite{BDH2022}. It would be interesting to determine if the outliers for the spectrum of random bi-regular graphs follow a similar staircase behaviour as a consequence of Theorem \ref{thm:co-growth2} \`{a} la Chen, Garza-Vargas, Tropp and van Handel.

Our proof of Theorem \ref{thm:co-growth2} differs in flavour to the approach of Puder, and is inspired by a modern proof of the classical co-growth formula due to Lyons and Peres \cite[Theorem 6.10]{LyonPeres2017}. The starting point is establishing an identity that connects a generalised resolvent involving powers of the adjacency matrix to the non-backtracking walk matrices (see equation \eqref{Jtransform2}). This identity serves as an operator-level co-growth formula which, with standard analytical techniques, can be used to derive the generalised co-growth formula for all non-negative functions. 

In fact, our approach also provides a simplified proof of the generalised co-growth formula for the regular tree (the first part of Theorem \ref{thm:co-growth}) which we demonstrate in Section \ref{sec:regular}. In Section \ref{sec:bi-regular} we then define the bi-resolvent for a bi-regular tree, connect it to the non-backtracking walk matrices and prove Theorem \ref{thm:co-growth2}.

By lifting, we also obtain the same generalised co-growth rate for any regular or bi-regular graph.

\begin{corollary}
The formula of Theorem \ref{thm:co-growth2} holds for any bi-regular graph $G$.
\end{corollary}

\begin{proof}
Lifting the function $f$ to a function $\tilde{f}$ on the universal covering bi-regular tree we readily obtain $a_r(f)=a_r(\tilde{f})$ and $b_r(f)=b_r(\tilde{f})$ for all $r\in\mathbb{N}$ where on the base we fix a vertex $v$ and on the cover we choose any fixed element in its fibre. The result is then immediate from Theorem \ref{thm:co-growth2}.
\end{proof}

\section{The regular tree}\label{sec:regular}

In this section $G$ will be the infinite $d$-regular tree with adjacency operator $A$. The powers of $A$ are related to the resolvent of $A$ by the following formula:
\begin{equation}\label{walkgenerating}
    (z-A)^{-1}=\sum_{r=0}^\infty z^{-r-1}A^r,
\end{equation}
where the absolute convergence (in operator norm) holds whenever $|z|>2\sqrt{d-1}=\|A\|$.

For any locally-finite graph, we can also define its related non-backtracking walk matrices $\{A_r\}_{r=0}^{\infty}$ (not to be confused with the Hashimoto matrix). Explicitly, we have
$$\left(A_r\right)_{i,j}=\# \{\text{non-backtracking walks from $i$ to $j$ of length $r$}\},$$
with the additional convention $A_0=I$. For a regular graph, it is direct to check from the recurrence relation of $\{A_r\}_{r=0}^{\infty}$ (See for example, \cite[Lemma 1.4.1, 1.4.2]{DSV2003} or Lemma \ref{lemma:NBWMrecurrence} of this note) that (formally) the generating function of these non-backtracking matrices is given by
\begin{equation}\label{nonbacktrackingwalkgenerating}
    (1-t^2)\left(\left(1+(d-1)t^2\right)I-tA\right)^{-1}=\sum_{r=0}^{\infty} t^rA_r.
\end{equation}
On the $d$-regular tree, the series on the right-hand side of \eqref{nonbacktrackingwalkgenerating} converges absolutely (in operator norm) when $|t|<1/\sqrt{d-1}$, see \eqref{eq:A_r-limsup} in the proof of Lemma \ref{lemma:NBWMrecurrence} specialized to $d_U=d_W=d$ for the sake of completeness. In fact, one can directly show

$$\lVert A_r\rVert\leq \mathrm{const}\cdot r(d-1)^{r/2},$$
which also confirms the absolute convergence.

Comparing the left hand side of \eqref{walkgenerating} and \eqref{nonbacktrackingwalkgenerating},   it follows that for any $|\rho|>\sqrt{d-1}$, as operators on $\ell^2(V(G))$,
\begin{equation}\label{Jtransform}
\sum_{r=0}^{\infty}  \left(\rho+\frac{d-1}{\rho}\right)^{-r-1}A^r=\left(\rho-\rho^{-1}\right)^{-1}\sum_{r=0}^{\infty} \rho^{-r}A_r.
\end{equation}

With these preparations, we are now ready to prove the first part of Theorem \ref{thm:co-growth}.

\begin{proof}[Proof of Theorem \ref{thm:co-growth} for the regular tree]
Fix an origin $e$ of $G$ and let $\delta_e$ be the indicator function at $e$. For any non-negative function $f$ on $G$, define its finitely supported truncation $f_r$  by
$$f_r(v)=
\begin{cases}
  f(v)& \mathrm{dist}(e,v) \leq r,\\
  0& \text{otherwise}.
\end{cases}
$$
In particular, the $f_r$ are in $\ell^2(V)$. By definition, $0\leq f_1 \leq f_2 \leq f_3\leq \cdots \leq f$ and 
$$a_r(f)=\langle f, A_r \delta_e \rangle\ \ \text{and}\ \ b_r(f)=\langle f, A^r \delta_e \rangle.$$
By equation (\ref{Jtransform}), for any $m\in\mathbb{N}$ and any $\rho > \sqrt{d-1}$, 
$$\left\langle f_m, \sum_{r=0}^{\infty} \left(\rho+\frac{d-1}{\rho}\right)^{-r-1} A^r\delta_e\right\rangle =\left\langle f_m, \left(\rho-\rho^{-1}\right)^{-1}\sum_{r=0}^{\infty} \rho^{-r}A_r \delta_e\right\rangle,$$
which is equivalent to
$$\sum_{r=0}^{\infty}b_r(f_m) \left(\rho+\frac{d-1}{\rho}\right)^{-r-1}=\left(\rho-\rho^{-1}\right)^{-1}\sum_{r=0}^{\infty} a_r(f_m)\rho^{-r}.$$
Notice that for any fixed $r\in \mathbb{N}$ and $m<n \in \mathbb{N}$, we have 
$$0\leq a_r(f_m)\leq a_r(f_n)\leq a_r(f)\ \ \text{and}\ \ 0\leq b_r(f_m)\leq b_r(f_n)\leq b_r(f).$$
Applying the monotone convergence theorem, 
\begin{align*}
    \sum_{r=0}^{\infty}b_r(f) \left(\rho+\frac{d-1}{\rho}\right)^{-r-1}&=\sum_{r=0}^{\infty}\lim_{m\rightarrow \infty}b_r(f_m) \left(\rho+\frac{d-1}{\rho}\right)^{-r-1}\\
    &=\lim_{m\rightarrow \infty} \sum_{r=0}^{\infty}b_r(f_m) \left(\rho+\frac{d-1}{\rho}\right)^{-r-1}\\
    &=\left(\rho-\rho^{-1}\right)^{-1}\lim_{m\rightarrow \infty} \sum_{r=0}^{\infty} a_r(f_m)\rho^{-r}\\
    &=\left(\rho-\rho^{-1}\right)^{-1}\sum_{r=0}^{\infty} \lim_{m\rightarrow \infty} a_r(f_m)\rho^{-r}\\
    &=\left(\rho-\rho^{-1}\right)^{-1}\sum_{r=0}^{\infty} a_r(f)\rho^{-r}.
\end{align*}
Note that both sides of $$\sum_{r=0}^{\infty}b_r(f) \left(\rho+\frac{d-1}{\rho}\right)^{-r-1}=\left(\rho-\rho^{-1}\right)^{-1}\sum_{r=0}^{\infty} a_r(f)\rho^{-r}$$ may be infinite.
However, if $\alpha(f)=\limsup_{r\rightarrow \infty}{a_r}(f)^{1/r}=\rho_0 \geq \sqrt{d-1}$, then for all $\rho>\rho_0$, $$\sum_{r=0}^{\infty}b_r(f) \left (\rho+\frac{d-1}{\rho}\right)^{-r-1}=\left(\rho-\rho^{-1}\right)^{-1}\sum_{r=0}^{\infty} a_r(f)\rho^{-r}<+\infty,$$ therefore $\limsup_{r\rightarrow \infty}{b_r(f)}^{1/r}\leq \rho_0+\frac{d-1}{\rho_0}$. Similarly if $\limsup_{r\rightarrow \infty}{b_r}(f)^{1/r}=\rho_0+\frac{d-1}{\rho_0}$ for some $\rho_0 \geq \sqrt{d-1}$ we have $\limsup_{r\rightarrow \infty}{a_r}(f)^{1/r}\leq \rho_0$. 

It remains to show that $\beta(f)\geq 2\sqrt{d-1}$ as long as $f$ is not identically zero. Let $v\in V(G)$ be such that $f(v)>0$. By the non-negativity of $f$, it suffices to show that $\beta(\delta_v)\geq 2\sqrt{d-1}$. Let $l$ be the distance between $e$ and $v$ in the tree, and fix a (in fact the only) non-backtracking path from $e$ to $v$ realising this distance. By definition,  
$$b_r(\delta_v)=\# \{\text{walks from $e$ to $v$ of length $r$}\}.$$
For any $r>l$, one can extend uniquely a walk from $e$ to $e$ of length $r-l$ to a walk from $e$ to $v$ by continuing the walk with the fixed path from $e$ to $v$. Thus we have 
$$b_r(\delta_v)\geq \# \{\text{walks from $e$ to $e$ of length $r-l$}\}=b_{r-l}(\delta_e),$$
and it follows that $\limsup_{r\rightarrow \infty}{b_r}(\delta_v)^{1/r}\geq \limsup_{r\rightarrow \infty}{b_r}(\delta_e)^{1/r}$. However, 
$$\beta(\delta_e)=\limsup_{r\rightarrow \infty}\langle \delta_e, A^r \delta_e \rangle ^{1/r}=\lVert A \rVert=2\sqrt{d-1},$$
as required.
\end{proof}
\begin{remark}
    By replacing $\rho$ by $-\rho$ in (\ref{Jtransform}), we can show that \begin{equation*}
\sum_{r=0}^{\infty} A^{2r} \left(\rho+\frac{d-1}{\rho}\right)^{-2r-1}=\left(\rho-\rho^{-1}\right)^{-1}\sum_{r=0}^{\infty} A_{2r} \rho^{-2r}.
\end{equation*}
\begin{equation*}
\sum_{r=0}^{\infty} A^{2r+1} \left(\rho+\frac{d-1}{\rho}\right)^{-2r-2}=\left(\rho-\rho^{-1}\right)^{-1}\sum_{r=0}^{\infty} A_{2r+1} \rho^{-2r-1}.
\end{equation*}
The same argument as above leads to other co-growth-type formulas in \cite[Theorem 2.1]{Puder2024}.
\end{remark}

\section{The bi-regular tree}\label{sec:bi-regular}

In this section, we prove Theorem \ref{thm:co-growth2}. Let $G=(V, E)$, $V=U \sqcup W$ be a $(d_U,d_W)$-bi-regular tree where $U$ and $W$ are the set of vertices of degree $d_U$ and $d_W$ respectively.
Denote its adjacency matrix as $A$ and degree matrix as $D$.  Define the projection matrices $I_U$ and $I_W$ such that 
$$\left(I_U\right)_{i,j}=
\begin{cases}
  1& i=j \in U,\\
  0& \text{otherwise},
\end{cases}
\,\,\,\,\, \text{and}\,\,\,\,\,
\left(I_W\right)_{i,j}= \begin{cases}
  1& i=j \in W,\\
  0& \text{otherwise}.
\end{cases}
$$
Then we have 
$$D=d_U I_U+d_W I_W.$$

To prove the co-growth formula for the bi-regular tree, we will study a generalisation of the resolvent which we call the \textit{bi-resolvent}. To this end, for any $z_1,z_2 \in \mathbb{C}$, let 
    $$Z=z_1I_U+z_2I_W.$$
The bi-resolvent of $A$ at $(z_1,z_2)$ is then defined as
    $$(Z-A)^{-1}$$
whenever $Z-A$ has a bounded $\ell^2$ inverse. 
\begin{lemma}
The bi-resolvent satisfies the identity
    \begin{align} \label{biresolvent-formula}
    (Z-A)^{-1}=\sum_{r=0}^\infty (z_1z_2)^{-r}A^{2r}Z^{-1} +\sum_{r=0}^\infty (z_1z_2)^{-(r+1)}A^{2r+1},
    \end{align}
where the series converge absolutely in operator norm whenever $|z_1z_2|>\|A\|^2=(\sqrt{d_U-1}+\sqrt{d_W-1})^2$.
\end{lemma}
\begin{proof}
Since the graph is bipartite on the sets $U$ and $W$, a direct computation readily shows that
    $$Z^{-1}AZ^{-1} = (z_1z_2)^{-1}A.$$
It follows that for any integer $r\geq 0$,
    \begin{align*}
        (Z^{-1}A)^{2r} = (Z^{-1}AZ^{-1}A)^r = (z_1z_2)^{-r}A^{2r},
    \end{align*}
and
    \begin{align*}
        (Z^{-1}A)^{2r+1}Z^{-1} = (Z^{-1}AZ^{-1}A)^rZ^{-1}AZ^{-1} = (z_1z_2)^{-(r+1)}A^{2r+1}.
    \end{align*}
Formally, we can write
    \begin{align*}
   (Z-A)^{-1}&=(1-Z^{-1}A)^{-1}Z^{-1}\\
    &=\sum_{r=0}^{\infty}(Z^{-1}A)^rZ^{-1}\\
    &=\sum_{r=0}^{\infty} (Z^{-1}A)^{2r}Z^{-1} + (Z^{-1}A)^{2r+1}Z^{-1}\\
    &=\sum_{r=0}^\infty (z_1z_2)^{-r}A^{2r}Z^{-1}+\sum_{r=0}^\infty (z_1z_2)^{-(r+1)}A^{2r+1}.
\end{align*}
Notice that absolute convergence of the series on the right hand-side holds whenever $|z_1z_2|>\|A\|^2=(\sqrt{d_U-1}+\sqrt{d_W-1})^2$, then it is direct to check by computation that it is indeed the $\ell^2$ inverse of $Z-A$.
\end{proof}

The recurrence relation of the non-backtracking walk matrices for bi-regular trees are slightly more complicated than the regular one:
\begin{lemma}\label{lemma:NBWMrecurrence}
The non-backtracking walk matrices $\{A_r\}_{r=0}^{\infty}$ for the bi-regular tree satisfy 
    \begin{enumerate}
        \item $A_rA=A_{r+1}+A_{r-1}(D-I)$ whenever $r\geq 2$,
        \item $A_1A = A_2+D$.
    \end{enumerate}
As a consequence,  
\begin{equation}\label{NBWbigenerating}
    (1-t^2)\left(I+t^2(D-I)-tA\right)^{-1}=\sum_{r=0}^{\infty} t^rA_r,
\end{equation}
where absolute convergence of the series holds in operator norm whenever $$|t|<\left(d_U-1\right)^{-1/4}\left(d_W-1\right)^{-1/4}$$.
\end{lemma}

\begin{proof}
Suppose first that $r\geq 2$. Then, $(A_rA)_{i,j}$ counts the number of paths of length $r+1$ from $i$ to $j$ such that the first $r$ steps are non-backtracking. There are two possibilities for such a path. First, the full path could be non-backtracking, of which there are $(A_{r+1})_{i,j}$ such paths. Otherwise, the final step is backtracking and so the path starts as a length $(r-1)$ non-backtracking walk from $i$ to $j$ of which there are $(A_{r-1})_{i,j}$ choices, followed by a non-backtracking step from $j$ of which there are $\mathrm{deg}(j)-1=(D-I)_{j,j}$ choices, followed by the single choice of returning over this edge back to vertex $j$. In total, there are $(A_{r-1})_{i,j}(D-I)_{j,j}=(A_{r-1}(D-I))_{i,j}$ choices for the second option, and we conclude the first claim of the lemma.

When $r=1$, $(A_1A)_{i,j}=A^2_{i,j}$ is counting the number of paths of length two from $i$ to $j$. If $i\neq j$ then this is simply the number of non-backtracking walks of length 2 from $i$ to $j$ and is hence $(A_2)_{i,j}$. On the other hand, if $i=j$ then the path is traversing an edge emanating from $i$ of which there are $\mathrm{deg}(i)=D_{i,i}$ choices, followed by the single choice to traverse back the same edge. Since $A_2$ is zero on the diagonal and $D$ is zero off the diagonal, the second claim of the lemma follows.

The equality of both sides of \eqref{NBWbigenerating} can be readily checked (formally) by making use of the aforementioned recurrence relations. The absolute convergence then holds whenever $|t|\limsup_{r\to\infty}\|A_r\|^\frac{1}{r}<1$. To compute this $\limsup$ we let $B$ be the non-backtracking adjacency operator, or Hashimoto matrix, acting on $\ell^2(\overrightarrow{E})$, where $\overrightarrow{E}$ is the set of directed edges of $G$, by
$$Bf(\overrightarrow{e_0})=\sum_{\overrightarrow{e}:\overrightarrow{e_0}\to\overrightarrow{e}}f(\overrightarrow{e}),$$
where $\overrightarrow{e_0}\to\overrightarrow{e}$ means $\mathrm{end}(\overrightarrow{e_0})=\mathrm{start}(\overrightarrow{e})$ and $\overrightarrow{e}$ is not the reversed direction of $\overrightarrow{e_0}$. In other words, $B$ is the matrix indexed by directed edges of the graph with entries
    $$B_{\overrightarrow{e_1},\overrightarrow{e_2}}=\begin{cases}
        1 & \text{if } \overrightarrow{e_1}\to\overrightarrow{e_2},\\
        0 & \text{otherwise.}
    \end{cases}$$
Moreover, an easy computation shows that 
    $$B^k_{\overrightarrow{e_1},\overrightarrow{e_2}} = \#\left\{\parbox{10cm}{non-backtracking walks of length $k+1$ from $\mathrm{start}(\overrightarrow{e_1})$ to $\mathrm{end}(\overrightarrow{e_2})$ starting with edge $\overrightarrow{e_1}$ and ending with edge $\overrightarrow{e_2}$.}\right\}. $$
Next, we let $S$ and $E$ be the incidence matrices whose rows are indexed by the vertices and directed edges of $G$ respectively, and whose columns are indexed by the directed edges and vertices of $G$ respectively and whose entries are given by
    \begin{align*}S_{v,\overrightarrow{e}}=\begin{cases}
        1 & \text{if } \mathrm{start}(\overrightarrow{e})=v,\\
        0 & \text{otherwise.}
    \end{cases} \hspace{.5cm}
    E_{\overrightarrow{e},v}=\begin{cases}
        1 & \text{if } \mathrm{end}(\overrightarrow{e})=v,\\
        0 & \text{otherwise.}
    \end{cases}
    \end{align*}
An easy computation then shows that 
    $$A_{r+1} = SB^rE.$$
We can also check that as an operator from $\ell^2(V)\to\ell^2(\overrightarrow{E})$ one has $\|E\|\leq\max\{d_U,d_W\}$. Similarly, as an operator from $\ell^2(\overrightarrow{E})\to\ell^2(V)$, one has $\|S\|\leq\max\{d_U,d_W\}$. 

Moreover, by \cite[Theorem 4.2]{AFH2015} the spectral radius of $B$ as an operator $\ell^2(\overrightarrow{E})\to\ell^2(\overrightarrow{E})$ is given by $\sqrt{\lim_{r\to\infty}|B(v,r)|^\frac{1}{r}}$, where $B(v,r)$ is the ball of radius $r$ centred at the vertex $v$ in $G$, and the quantity is independent of the choice of $v$ in $G$. Without loss of generality, assume that the degree of $v$ is $d_U$, then the number of vertices of $G$ that are at a distance of $r$ from $v$ is
    $$
    \begin{cases}
        d_U\left(d_W-1\right)^{\frac{r}{2}}\left(d_U-1\right)^{\frac{r}{2}-1} & \text{if $r\geq 2$ is even}\\
        d_U\left(d_W-1\right)^{\frac{r-1}{2}}\left(d_U-1\right)^{\frac{r-1}{2}} & \text{if $r\geq 2$ is odd}
    \end{cases} \leq d_U\left(d_W-1\right)^\frac{r}{2}\left(d_U-1\right)^{\frac{r}{2}}.$$
It follows that 
    $$\sqrt{\lim_{r\to\infty}|B(v,r)|^\frac{1}{r}} \leq \left(d_W-1\right)^\frac{1}{4}\left(d_U-1\right)^{\frac{1}{4}}\sqrt{\lim_{r\to\infty}\left(d_Ur\right)^\frac{1}{r}}=\left(d_W-1\right)^\frac{1}{4}\left(d_U-1\right)^{\frac{1}{4}}.$$
Putting everything together, we obtain
    \begin{align}
    \label{eq:A_r-limsup}
    \limsup_{r\to\infty}\|A_r\|^\frac{1}{r}&=\limsup_{r\to\infty}\|SB^rE\|^\frac{1}{r} \nonumber\\
    &\leq \limsup_{r\to\infty}\|S\|^\frac{1}{r}\limsup_{r\to\infty}\|E\|^\frac{1}{r}\limsup_{r\to\infty}\|B^r\|^\frac{1}{r} \\
    &=\mathrm{spec\ rad}(B)=\left(d_W-1\right)^\frac{1}{4}\left(d_U-1\right)^{\frac{1}{4}}, \nonumber
    \end{align}
where the penultimate equality follows from the Gelfand formula and the bounds on the operator norms of $S$ and $E$. 
\end{proof}

For any $|\rho|>\left(d_U-1\right)^{1/4}\left(d_W-1\right)^{1/4}$ if we set $z_1=\rho+\left(d_U-1\right)/\rho$ and $z_2=\rho+\left(d_W-1\right)/\rho$, then we have
    $$|z_1z_2|>(\sqrt{d_U-1}+\sqrt{d_W-1})^2.$$
Thus by comparing the left hand side of \eqref{biresolvent-formula} and \eqref{NBWbigenerating},
\begin{equation}\label{Jtransform2}
\begin{aligned}
    \left(\rho-\rho^{-1}\right)^{-1}\sum_{r=0}^{\infty} \rho^{-r}A_r&=\sum_{r=0}^\infty (z_1z_2)^{-r}A^{2r}Z^{-1}+\sum_{r=0}^\infty (z_1z_2)^{-(r+1)}A^{2r+1}
\end{aligned}
\end{equation}
as operators on $\ell^2(V(G))$.

\begin{proof}[Proof of Theorem \ref{thm:co-growth2}]
Fix an origin $e$ of $G$ and let $\delta_e$ be the indicator function at $e$. Without loss of generality, assume that $e\in U$. As in Section \ref{sec:regular}, for any non-negative function $f$ on $G$, $f_r$ will denote its finitely supported truncation at distance $r$ from $e$.
By equation (\ref{Jtransform2}), for any $m\in\mathbb{N}$ and any $\rho > \left(d_U-1\right)^{1/4}\left(d_W-1\right)^{1/4}$, 
\begin{align*}&\left\langle f_m, \sum_{r=0}^\infty (z_1z_2)^{-r}A^{2r}Z^{-1}\delta_e+\sum_{r=0}^\infty (z_1z_2)^{-(r+1)}A^{2r+1}\delta_e\right\rangle\\
&\hspace{3cm}=\left\langle f_m, \left(\rho-\rho^{-1}\right)^{-1}\sum_{r=0}^{\infty} \rho^{-r}A_r \delta_e\right\rangle,
\end{align*}
which, since $Z^{-1}\delta_e=z_1^{-1}\delta_e$, is equivalent to
$$z_1^{-1}\sum_{r=0}^{\infty}\frac{b_{2r}(f_m)}{(z_1z_2)^{r}}+\sum_{r=0}^{\infty}\frac{b_{2r+1}(f_m)}{(z_1z_2)^{r+1}} =\left(\rho-\rho^{-1}\right)^{-1}\sum_{r=0}^{\infty} a_r(f_m)\rho^{-r}.$$
Notice that for any fixed $r\in \mathbb{N}$ and $m<n \in \mathbb{N}$, we have 
$$0\leq a_r(f_m)\leq a_r(f_n)\leq a_r(f)\ \ \text{and}\ \ 0\leq b_r(f_m)\leq b_r(f_n)\leq b_r(f).$$
Thus by applying the monotone convergence theorem as in the proof of Theorem \ref{thm:co-growth}, we arrive at the equality 
    \begin{align}
        \label{eq:equivalence}
        z_1^{-1}\sum_{r=0}^{\infty}\frac{b_{2r}(f)}{(z_1z_2)^{r}}+\sum_{r=0}^{\infty}\frac{b_{2r+1}(f)}{(z_1z_2)^{r+1}} =\left(\rho-\rho^{-1}\right)^{-1}\sum_{r=0}^{\infty} a_r(f)\rho^{-r},
    \end{align}
where both sides may be infinite.

However, if $\alpha(f)=\limsup_{r\rightarrow \infty}{a_r}(f)^{1/r}=\rho_0 \geq \left(d_U-1\right)^{1/4}\left(d_W-1\right)^{1/4}$, then for all $\rho>\rho_0$, the right-hand side, and hence the left-hand side, of \eqref{eq:equivalence} are finite. But, since $z_1,z_2\geq 1$, we then have
    $$(z_1z_2)^{-1}\sum_{r=0}^\infty \frac{b_r(f)}{(z_1z_2)^\frac{r}{2}}\leq z_1^{-1}\sum_{r=0}^{\infty}\frac{b_{2r}(f)}{(z_1z_2)^{r}}+\sum_{r=0}^{\infty}\frac{b_{2r+1}(f)}{(z_1z_2)^{r+1}} < \infty,$$
and so $\limsup_{r\to\infty}b_r(f)^{\frac{1}{r}}\leq \left(\rho_0+\left(d_U-1\right)/\rho_0\right)^{1/2}\left(\rho_0+\left(d_W-1\right)/\rho_0\right)^{1/2}$. 

Conversely, if $\limsup_{r\to\infty}b_r(f)^{\frac{1}{r}}= \left(\rho_0+\left(d_U-1\right)/\rho_0\right)^{1/2}\left(\rho_0+\left(d_W-1\right)/\rho_0\right)^{1/2}$ for some $\rho_0\geq \left(d_U-1\right)^{1/4}\left(d_W-1\right)^{1/4}$, then for any $\rho>\rho_0$, we have 
    \begin{align*}
        \limsup_{r\to\infty}\left(b_{2r}(f)(z_1z_2)^{-r}\right)^\frac{1}{r}&<1,\\
        \limsup_{r\to\infty}\left(b_{2r+1}(f)(z_1z_2)^{-r}\right)^\frac{1}{r}&<1,
    \end{align*}
so that both sides of \eqref{eq:equivalence} are finite and $\limsup_{r\rightarrow \infty}{a_r}(f)^{1/r}\leq\rho_0$. 

It remains to show that $\limsup_{r\to\infty}b_r(f)^{\frac{1}{r}}\geq \sqrt{d_U-1}+\sqrt{d_W-1}$, where the right-hand side of this expression is precisely $(\rho+\left(d_U-1\right)/\rho)^{1/2}(\rho+\left(d_W-1\right)/\rho)^{1/2}$ with $\rho=\left(d_U-1\right)^{1/4}\left(d_W-1\right)^{1/4}$. But, by an identical argument to the end of the proof of Theorem \ref{thm:co-growth}, we obtain that 
    $$\limsup_{r\to\infty}b_r(f)^{\frac{1}{r}}\geq \|A\|= \sqrt{d_U-1}+\sqrt{d_W-1},$$
as required.

\end{proof}

\appendix

\section{Vertex independence for $\alpha(f)$ and $\beta(f)$}\label{Appendix}

In this section, we prove that for any connected graph $G$, the growth rates $\alpha(f)$ and $\beta(f)$ are independent of the choice of vertex $v$. To this end, for $v$ a vertex in $G$, we let $a_r(f;v)$ and $b_r(f;v)$ be defined as $a_r(f)$ and $b_r(f)$ but with the dependence on $v$ made explicit. 
Now note that if $\Tilde{G}$ is the universal cover of $G$ and $\Tilde{f}$ the pull-back of $f$ under the universal covering map, then for any vertex $\Tilde{v}$ in fiber of $v$, we have $$a_r(f;v)=a_r(\Tilde{f};\Tilde{v})\,\,\,\,\,\text{and}\,\,\,\,\,b_r(f;v)=a_r(\Tilde{f};\Tilde{v}).$$
Thus, it suffices to consider when $G$ is a tree.

\begin{proposition}
For any vertices $u,v$ in a tree $T$ we have 
\begin{align*}
    \limsup_{r\rightarrow \infty}{a_r(f;u)^{\frac{1}{r}}} &= \limsup_{r\rightarrow \infty}{a_r(f;v)^{\frac{1}{r}}},\\
    \limsup_{r\rightarrow \infty}{b_r(f;u)^{\frac{1}{r}}} &= \limsup_{r\rightarrow \infty}{b_r(f;v)^{\frac{1}{r}}}.
\end{align*}
\end{proposition}
\begin{proof}
Clearly it suffices to prove the result in the case where $u$ and $v$ are adjacent in $T$. Let $T'$  be the subgraph of $T$ obtained by deleting the edge connecting $u$ and $v$. Note that $T'$ is a forest with two components and we have the following obvious relations
$$a_r(f;v)=a_r(f|_{T'};v)+a_{r-1}(f|_{T'};u),$$
$$a_r(f;u)=a_r(f|_{T'};u)+a_{r-1}(f|_{T'};v),$$
$$b_r(f;v)\geq b_{r-1}(f;u),$$
$$b_r(f;u)\geq b_{r-1}(f;v).$$
It is then clear that the result follows for the $b_r$. Moreover, the relations for the $a_r$ mean that both $\limsup_{r\rightarrow \infty}{a_r(f;u)^{\frac{1}{r}}}$ and $\limsup_{r\rightarrow \infty}{a_r(f;v)^{\frac{1}{r}}}$ are equal to
$$
\begin{aligned}
    \max \{\limsup_{r\rightarrow \infty}{a_r(f|_{T'};u)^{\frac{1}{r}}}, \limsup_{r\rightarrow \infty}{a_r(f|_{T'};v,)^{\frac{1}{r}}}\},\\
\end{aligned}
$$
and hence equal to one another, as desired.
\end{proof}

\section*{Acknowledgments}

WL is funded by the National Natural Science Foundation of China (Grant No. 123B2013). JT is funded by the Leverhulme Trust through a Leverhulme Early Career Fellowship (Grant No. ECF-2024-440). This work originated from private communications between the authors and Doron Puder. We extend our gratitude to Doron for bringing this problem to our attention. Furthermore, WL thanks Charles Bordenave, Benoît Collins, and Shiping Liu for arranging WL’s scholarly visit to Charles, which led to meeting JT and forming a lasting friendship and collaboration.

\bibliographystyle{abbrv}

  \begin{tabular}{@{}l@{}}%
    \text{Wenbo Li}\\
    \text{School of Mathematical Sciences,}\\
    \text{University of Science and Technology of China,}\\
    \text{No.96 Jinzhai Road, Hefei,}\\
    \text{China}\\
    \texttt{patlee@mail.ustc.edu.cn}
  \end{tabular}
\ \\ \ \\
  
\begin{tabular}{@{}l@{}}%
    \text{Joe Thomas}\\
    \text{Department of Mathematical Sciences,}\\
    \text{Durham University,}\\
    \text{Lower Mountjoy, DH1 3LE Durham,}\\
    \text{United Kingdom}\\
    \texttt{joe.thomas@durham.ac.uk}
  \end{tabular}

\end{document}